\documentclass{amsart}

\usepackage{amsmath,amssymb,graphicx}

\newtheorem{thm}{Theorem}
\newtheorem{lem}[thm]{Lemma}

\newtheorem{prop}[thm]{Proposition}

\theoremstyle{definition}
\newtheorem{defn}[thm]{Definition}

\newtheorem{rmk}[thm]{Remark}
\newtheorem{exm}[thm]{Example}

\newcommand{\CPb}{\overline{\mathbb{CP}}{}^{2}}
\newcommand{\CP}{{\mathbb{CP}}{}^{2}}
\newcommand{\R}{\mathbb{R}}

\newcommand{\scparallel}{{\scriptscriptstyle \parallel}}

\title[Note on new symplectic $4$-manifolds nonnegative signature ]
{Note on new symplectic $4$-manifolds with nonnegative signature \\} 

\begin{document}

\author{Anar Akhmedov}
\address{School of Mathematics, 
University of Minnesota, 
Minneapolis, MN, 55455, USA}
\email{akhmedov@math.umn.edu}

\date{November 12, 2011. Revised on April 2012}


\begin{abstract} In this short note, we present a construction of new symplectic $4$-manifolds with non-negative signature using the complex surfaces on Bogomolov-Miyaoka-Yau line $c_1^2 = 9\chi_h$, the fake projective planes and Cartwright-Steger surfaces. Our construction yields an infinite family of fake rational homology $(2n-1)\CP\#(2n-1)\CPb$ for any integer $3 \leq n \leq 22$. 
\end{abstract}

\maketitle

\section{Introduction}

This paper is the continuation of author's previous work on the geography of symplectic $4$-manifolds (\cite{A1}, \cite{akhmedov}, \cite{AP1}, \cite{ABBKP}, \cite{AP2}, \cite{AP3}, \cite{AP4}, \cite{AGP}). The purpose of this article is to construct new symplectic $4$-manifolds that are interesting with respect to the symplectic geography problem. Starting from the fake projective planes and Cartwright-Steger surfaces on Bogomolov-Miyaoka-Yau line $c_1^2 = 9\chi_h$, forming their symplectic sum with the product manifolds $\Sigma_{g} \times \Sigma_{h}$, and applying the sequence of Luttinger surgeries along the lagrangian tori, we obtain a family of symplectic $4$-manifolds with non-negative signature and with trivial first rational homology group. As a consequence of our technique, we produce an infinite family of fake symplectic and an infinite family of fake nonsymplectic rational homology $(2n-1)\CP\#(2n-1)\CPb$ for any integer $n\geq 3$. These examples are most interesting when $3 \leq n \leq 22$. We hope that the approach presented here is promising in constructing the exotic smooth structures on $4$-manifolds with nonnegative signature. In fact, there are strong reasons to expect that some of the manifolds produced in Theorem~\ref{thm:main} can be made simply connected by choosing the gluing diffeomorphism of the fiber sums carefully. Another potential application would be to construct the fake symplectic $\CP$'s using the fake projective planes and Cartwright-Steger surface with $c_1^2 = 9\chi_h = 9$. We will return to these points and other applications of these building blocks in a future work. The results and ideas outlined in this article were known to the author for some time, and has been communicated with some of his colleagues.   

Given two $4$-manifolds, $X$\/ and $Y$, we denote their connected sum by $X\# Y$. For a positive integer $m\geq 2$, the connected sum of $m$\/ copies of $X$\/ will be denoted by $mX$\/. Let $\CP$ denote the complex projective plane and let $\CPb$ denote the underlying smooth $4$-manifold $\CP$ equipped with the opposite orientation. Our main results are the following.

\begin{thm}\label{thm:main}
Let $M$ be one of the following\/ $4$-manifolds.  
\begin{itemize}
%
\item[(i)] $5\CP\#5\CPb$,

\vspace{3pt}
\item[(ii)] $(2n-1)\CP\#(2n-1)\CPb$ for any integer $4 \leq n \leq 22$.

\end{itemize}
Then there exist an infinite family of irreducible symplectic\/ and an infinite family of irreducible non-symplectic\/ $4$-manifolds, all of which have the rational homology of $M$. Furthermore, our examples have odd intersection form if $n$ is odd.

\end{thm}

Recall that exotic irreducible smooth structures on $(2n-1)\CP\#(2n-1)\CPb$ for $n\geq 23$ were constructed in \cite{AP3}. Although our construction holds for an arbitrary $n \geq 4$, we only state it in the case $3 \leq n \leq 22$, when it is most interesting.

Our paper is organized as follows. In Sections~\ref{sec:Luttinger}--\ref{sec: model}, we state some background results needed in this paper and collect building blocks that are needed in our construction of symplectic $4$-manifolds.
In Sections~\ref{sec:fake}, \ref{sec:faken}, we present proof of our main results. In Sections~\ref{sec:fake3}, we present new construction of fake symplectic $3(\mathbb{S}^2\times \mathbb{S}^2)$ using Mumford's fake projective plane.

\section{Luttinger surgery and symplectic cohomology $(2n-3)(\mathbb{S}^2\times \mathbb{S}^2)$}

\label{sec:Luttinger}

In the section we will briefly review a Luttinger surgery. For the details, we refer the reader to \cite{lu} and \cite{ADK}. Luttinger surgery has been very effective tool recently for constructing exotic smooth structures on $4$-manifolds.

\begin{defn} Let $X$\/ be a symplectic $4$-manifold with a symplectic form $\omega$, and the torus $\Lambda$ be a Lagrangian submanifold of $X$ with self-intersection $0$. Given a simple loop $\lambda$ on $\Lambda$, let $\lambda'$ be a simple loop on $\partial(\nu\Lambda)$ that is parallel to $\lambda$ under the Lagrangian framing. For any integer $k$, the $(\Lambda,\lambda,1/k)$ \emph{Luttinger surgery}\/ on $X$\/ will be 
$X_{\Lambda,\lambda}(1/k) = ( X - \nu(\Lambda) ) \cup_{\phi} (S^1 \times S^1 \times D^2)$,  the $1/k$\/ surgery on $\Lambda$ with respect to $\lambda$ under the Lagrangian framing. Here 
$\phi : S^1 \times S^1 \times \partial D^2 \to \partial(X - \nu(\Lambda))$ denotes a gluing map satisfying $\phi([\partial D^2]) = k[{\lambda'}] + [\mu_{\Lambda}]$ in $H_{1}(\partial(X - \nu(\Lambda))$, where $\mu_{\Lambda}$ is a meridian of $\Lambda$.

\end{defn}

It is  shown in \cite{ADK} that $X_{\Lambda,\lambda}(1/k)$ possesses a symplectic form that restricts to the original symplectic form $\omega$ on $X\setminus\nu\Lambda$. The following lemma is easy to verify and the proof is left as an exercise to the reader.

\begin{lem} $\pi_1(X_{\Lambda,\lambda}(1/k)) = \pi_1(X- \Lambda)/N(\mu_{\Lambda} \lambda'^k)$.
\smallskip
\item $\sigma(X)=\sigma(X_{\Lambda,\lambda}(1/k))$ and $e(X)=e(X_{\Lambda,\lambda}(1/k))$.
\end{lem}



\subsection{Construction of cohomology $(2n-3)(\mathbb{S}^2\times \mathbb{S}^2)$}
\label{sec:(2n-3)(S^2 x S^2)}

In this section, we recall the family of symplectic cohomology $(2n-3)(\mathbb{S}^2\times \mathbb{S}^2)$ constructed in \cite{FPS, AP1}, respectively for $n=2$, and $n \geq 3$. We also refer the reader to \cite{A1} where such family of 
symplectic $4$-manifolds were initially studied by the author using the knot surgery and the twisted fiber sum. 

Recall that for each integer $n\geq2$, there is family of irreducible pairwise non-diffeomorphic 4-manifolds $\{Y_n(m)\mid m=1,2,3,\dots\}$ that have the same integer cohomology ring as $(2n-3)(\mathbb{S}^2\times \mathbb{S}^2)$. $Y_n(m)$ are obtained by performing $2n+3$ Luttinger surgeries (cf.\ \cite{ADK, lu}) and a single $m$\/ torus surgery on $\Sigma_2\times \Sigma_n$. These $2n+4$ torus surgeries are performed as follows

\begin{eqnarray}\label{first 8 Luttinger surgeries}
&&(a_1' \times c_1', a_1', -1), \ \ (b_1' \times c_1'', b_1', -1), \ \
(a_2' \times c_2', a_2', -1), \ \ (b_2' \times c_2'', b_2', -1),\\ \nonumber
&&(a_2' \times c_1', c_1', +1), \ \ (a_2'' \times d_1', d_1', +1),\ \
(a_1' \times c_2', c_2', +1), \ \ (a_1'' \times d_2', d_2', +m),
\end{eqnarray}
together with the following $2(n-2)$ additional Luttinger surgeries
\begin{gather*}
(b_1'\times c_3', c_3',  -1), \ \ 
(b_2'\times d_3', d_3', -1), \  \dots  ,\ 
(b_1'\times c_n', c_n',  -1), \ \
(b_2'\times d_n', d_n', -1).
\end{gather*}
Here, $a_i,b_i$ ($i=1,2$) and $c_j,d_j$ ($j=1,\dots,n$) denote the standard loops that generate $\pi_1(\Sigma_2)$ and $\pi_1(\Sigma_n)$, respectively. See Figure~\ref{fig:lagrangian-pair} for a typical Lagrangian tori along which the surgeries are performed.  

\begin{figure}[ht]
\begin{center}
\includegraphics[scale=.49]{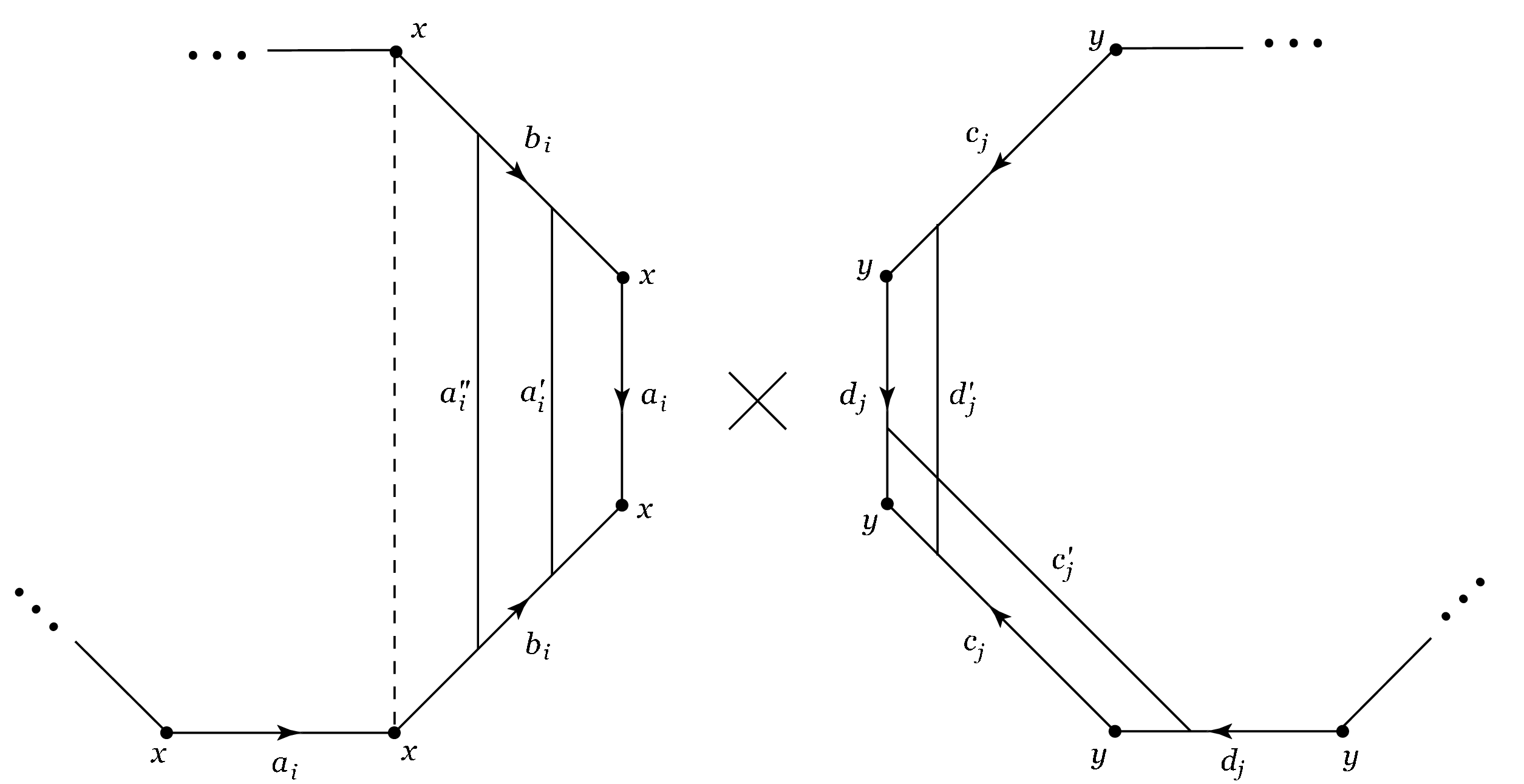}
\caption{Lagrangian tori $a_i'\times c_j'$ and $a_i''\times d_j'$}
\label{fig:lagrangian-pair}
\end{center}
\end{figure}

Since $m$-torus surgery is non-symplectic for $m\geq 2$, the manifold $Y_n(m)$ is symplectic only when $m=1$. The Euler characteristic of $Y_n(m)$ is $4n-4$ and its signature is $0$.  $\pi_1(Y_n(m))$ is generated by $a_i,b_i,c_j,d_j$ ($i=1,2$ and $j=1,\dots,n$) and the following relations hold in $\pi_1(Y_n(m))$:  
\begin{gather}\label{Luttinger relations}
[b_1^{-1},d_1^{-1}]=a_1,\ \  [a_1^{-1},d_1]=b_1,\ \  [b_2^{-1},d_2^{-1}]=a_2,\ \  [a_2^{-1},d_2]=b_2,\\ \nonumber
[d_1^{-1},b_2^{-1}]=c_1,\ \ [c_1^{-1},b_2]=d_1,\ \ [d^{-1}_2,b^{-1}_1]=c_2,\ \ [c_2^{-1},b_1]^m=d_2,\\ \nonumber
 [a_1,c_1]=1, \ \ [a_1,c_2]=1,\ \  [a_1,d_2]=1,\ \ [b_1,c_1]=1,\\ \nonumber
[a_2,c_1]=1, \ \ [a_2,c_2]=1,\ \  [a_2,d_1]=1,\ \ [b_2,c_2]=1,\\ \nonumber
[a_1,b_1][a_2,b_2]=1,\ \ \prod_{j=1}^n[c_j,d_j]=1,\\ \nonumber
[a_1^{-1},d_3^{-1}]=c_3, \ \ [a_2^{-1} ,c_3^{-1}] =d_3, \  \dots, \ 
[a_1^{-1},d_n^{-1}]=c_n, \ \ [a_2^{-1} ,c_n^{-1}] =d_n,\\ \nonumber
[b_1,c_3]=1,\ \  [b_2,d_3]=1,\ \dots, \
[b_1,c_n]=1,\ \ [b_2,d_n]=1.
\end{gather}

The surfaces $\Sigma_2\times\{{\rm pt}\}$ and $\{{\rm pt}\}\times \Sigma_n$ in $\Sigma_2\times\Sigma_n$ descend to surfaces in
$Y_n(m)$. They are symplectic submanifolds in $Y_n(1)$. We will denote their images by $\Sigma_2$ and $\Sigma_n$. 
Note that $[\Sigma_2]^2=[\Sigma_n]^2=0$ and $[\Sigma_2]\cdot[\Sigma_n]=1$. Let $\mu(\Sigma_2)$ and $\mu(\Sigma_n)$ 
denote the meridians of these surfaces in $Y_n(m)$. This construction easily generalizes to $\Sigma_3\times\Sigma_n$.
We will denote the resulting smooth manifold in this case as $Z_n(m)$. 

In our construction, we will also need the following result \cite{Sm}. 

\begin{prop} The homology class $2[\Sigma_{2} \times \{pt\}]$ in the four-manifold $\Sigma_{2} \times \mathbb{S}^2$ can be represented by connected symplectic genus $3$ surface.
\end{prop}

\section{Complex surfaces on Bogomolov-Miyaoka-Yau line}
\label{sec: model}

\subsection{Fake projective planes}
\label{sec: model1}

A fake projective plane is a smooth complex surface which is not the complex projective plane, but has the same Betti numbers as the complex projective plane. The first fake projective plane was constructed by David Mumford in 1979 using p-adic uniformization \cite{Mu}. He also showed that there could only be a finite number of such surfaces. Two more examples were found by Ishida and Kato \cite{IsKa} in 1998, and another by Keum \cite{Ke} in 2006. In their 2007 Inventiones paper \cite{PY} (see also Addendum \cite{PY1}), Gopal Prasad and Sai-Kee Yeung almost completely classified fake projective planes by proving that they fall into a small number of “classes”. Using the arithmeticity of the fundamental group of fake projective planes, and the formula for the covolume of principal arithmetic subgroups, they found twenty eight distinct classes of fake projective planes. For a very small number of classes, they left open the question of existence of fake projective planes in that class, but conjectured that there are none. Finally, Donald Cartwright and Tim Steger verified their conjecture and found all the fake projective planes, up to isomorphism, in each of the 28 classes \cite{CS}.

\begin{exm} In this example, we recall some properties of Mumford's fake projective plane $M$. We will use $M$ in Section~\ref{sec:fake}, but our construction works equally well with most of the fake projective planes . We refer the reader to \cite{Ye}, where a complete classification of all smooth surface of general type with Euler number $3$ are given. The Euler characteristic and the Betti numbers of $M$ are $e(M)=3$, $b_{1}(M)=0$ and $b_{2}(M)=1$. $M$ is a minimal complex surface of general type with $\sigma =1$, ${c_{1}}^2=3e = 9$ and $\chi_{h} = 1$. The intersection form of $M$ is odd, and isomorphic to $(1)$. The fundamental group $\Pi$ of $M$ is a torsion-free cocompact arithmetic subgroup of $PU(2, 1)$, thus $M$ is a ball quotient $B_{\mathbb{C}}^{2}/\Pi$. The canonical line bundle $K_{M}$ is divisible by $3$, i.e., there is a holomorphic line bundle $L$ on $M$ such that $K_{M} = 3L$. Also note the class of $L$ can be represented by a symplectic surface $H$ of self-intersection $1$ (see discussion in \cite{Ko}, pages 212-213). Using the adjunction formula, we compute the genus of the symplectic surface $H$: $g(H) = 1 + 1/2(H\cdot H + 3H \cdot H) = 3$. By symplectically blowing up $H$, we obtain a symplectic genus 3 surface with self-intersection 0 in $\bar{\Sigma}_3$ in $M\#\CPb$.
\end{exm}

\subsection{Complex surfaces of Cartwright and Steger}
\label{sec: model2}

The study of enumerating the set of all fake projective planes in the class $\mathcal{C}_{11}$ led Donald Cartwright and Tim Steger to discover a complex surface with irregularity $q=1$ and Euler characteristic $e=3$. Cartwright and Steger showed that a certain maximal arithmetic subgroup $\bar{\Gamma}$ of $PU(2; 1)$ contains a torsion-free subgroup $\Pi$ of index $864$ which has abelianization $\mathbb{Z}^2$. Such subgroup $\Pi$ is unique up to conjugation. Furthermore, for each ineger $n \geq 1$, $\Pi$ has a normal subgroup $\Pi_{n}$ of index $n$. Let $M_{n} = B^{2}(\mathbb{C})/\Pi_{n}$ denote the quotient of a complex hyperbolic space by a torsion free lattice $\Pi_{n}$ of $PU(2; 1)$. The Euler characteristic of $M_{n}$\/ is $e(M_{n})=ne(M_{1})=3n$. $M_{n}$\/ is a minimal complex surface of general type with $\sigma = n$, ${c_{1}}^2=3e = 9n$ and $\chi_{h} = n$. The intersection form of $M_{1}$ is odd, indefinite and isomorphic to $3(1) \oplus 2(-1)$. The Betti numbers of $M_{1}$ are: $1, 2, 5, 2, 1$.

For the convenience of the reader, let us first recall some details of their construction. Let $k = \mathbb{Q}(\sqrt{3})$ and $l = \mathbb{Q}(\zeta)$, where $\zeta  = e^{2{\pi}i/12}$ denote a primitive 12th roots of unity. Since $2\zeta - \zeta^{3} = \sqrt{3}$, $k$ is the subfield of $l$. Let $A$ denote the following matrix with entries in $k$:

\[
A =
\left[ {\begin{array}{ccc}
-1-\sqrt{3} & 1 & 0  \\
 1 & 1-\sqrt{3} & 0  \\
 0 & 0 & 1  \\
\end{array} } \right]
\]

Let $\bar{\Gamma} = \{ \xi \in M_{3,3}(\mathbb{Z}[\zeta]) : {\xi}^{*}A\xi = A \}$ modulo scalars. $\bar{\Gamma}$ is generated by the following four matrices:

\[
u =
\left[ {\begin{array}{ccc}
1 & 0 & 0  \\
-{\zeta}^{3} - {\zeta}^{2} + {\zeta} + 1 & {\zeta}^{3} & 0  \\
 0 & 0 & 1  \\
\end{array} } \right]
\]
,

\[
v =
\left[ {\begin{array}{ccc}
{\zeta}^{3} + 1 & {\zeta}^{3} - {\zeta}^{2} - {\zeta} + 1 & 0  \\
 {\zeta}^{2} + {\zeta} & -{\zeta}^{3} - 1 & 0  \\
 0 & 0 & 1  \\
\end{array} } \right]
\]
,

\[
j =
\left[ {\begin{array}{ccc}
{\zeta} & 0 & 0  \\
 0 & {\zeta} & 0  \\
 0 & 0 & 1  \\
\end{array} } \right]
\]
,

\[
b =
\left[ {\begin{array}{ccc}
{\zeta}^{3} + {\zeta}^{2} & -{\zeta}^{2} & {\zeta}^{2}-1  \\
 {\zeta}^{3} + 2{\zeta}^{2} + {\zeta} & -{\zeta} & {\zeta}^{3} + {\zeta}^{2} \\
 -{\zeta}^{3} - {\zeta}^{2} + {\zeta} + 1 & {\zeta}^{3} & -{\zeta}^{3} + {\zeta} + 1 \\
\end{array} } \right]
\]

According to Cartwright and Steger, $\pi_1(M_{1})$ is generated by $u$, $v$, $j$, and $b$ and the following relations hold in $\pi_1(M_{1})$:  

\begin{gather}\label{CS relations}
vubj = u,\ \  bj^{2} = ju,\ \  u^{2}vbu = j^{2}.\\ \nonumber
\end{gather}

\section{Construction of fake rational homology $5\CP\#5\CPb$} 
\label{sec:fake}

Let us fix a triple of integers $m \geq 1$, $p \geq 1$ and $q \geq 1$. Let $Y_1(1/p,m/q)$ denote smooth $4$-manifold obtained by performing the following 6 torus surgeries on $\Sigma_3\times \mathbb{T}^2$:

\begin{equation}\label{eq: Luttinger surgeries for Y_1(m)} 
(a_1' \times c', a_1', -1), \ \ (b_1' \times c'', b_1', -1),\\  \nonumber
(a_2' \times c', a_2', -1), \ \ (b_2' \times c'', b_2', -1),\\  \nonumber
(a_3' \times c', c', +1/p), \ \ (a_3'' \times d', d', +m/q).
\end{equation}

Here, $a_i,b_i$ ($i=1,2,3$) and $c,d$\/ denote the standard generators of $\pi_1(\Sigma_3)$ and $\pi_1(\mathbb{T}^2)$, respectively. Since all the torus surgeries above are Luttinger surgeries when $m = 1$, $Y_1(1/p,1/q)$ is a minimal symplectic 4-manifold. The fundamental group of $Y_1(1/p,m/q)$ is generated by $a_i,b_i$ ($i=1,2,3$) and $c,d$, and the following relations hold in $\pi_1(Y_1(1/p,m/q))$:

\begin{gather}\label{Luttinger relations for Y_1(m)}
[b_1^{-1},d^{-1}]=a_1,\ \  [a_1^{-1},d]=b_1,\ \
[b_2^{-1},d^{-1}]=a_2,\ \  [a_2^{-1},d]=b_2,\\ \nonumber
[d^{-1},b_3^{-1}]=c^p,\ \ {[c^{-1},b_3]}^{-m}=d^q,\\ \nonumber
[a_1,c]=1,\ \  [b_1,c]=1,\ \ [a_2,c]=1,\ \  [b_2,c]=1,\\ \nonumber
[a_3,c]=1,\ \  [a_3,d]=1,\\ \nonumber
[a_1,b_1][a_2,b_2][a_2,b_2]=1,\ \ [c,d]=1.
\end{gather}

Let $\Sigma_3 \subset Y_1(1/p,m/q)$ be a genus 3 surface that desend from the surface $\Sigma_3\times\{{\rm pt}\}$ in $\Sigma_3\times \mathbb{T}^2$.

Next, we take the normal connected sum 
\begin{equation*}
X_1(m, p, q, \psi)=Y_{1}(1/p,m/q)\#_{\psi}(M\#\CPb)
\end{equation*}

To perform our gluing, we can use any orientation reversing diffeomorphism $\psi:\partial(\nu\Sigma_{3})\rightarrow \partial(\nu\bar{\Sigma}_3)$ that restricts to orientation preserving diffeomorphism on parallel genus 3 surfaces and complex conjugation on the meridian circles. The homotopy type of $X_1(m, p, q, \psi)$ does depend on $p$, $q$, and $\psi$.    

Let us choose a base point $x$ of $\pi_1(Y_1(1/p,m/q))$ on $\partial(\nu\Sigma_3)$ such that $\pi_1(Y_1(1/p,m/q)\setminus\nu\Sigma_3,x)$ is normally generated by $a_i,b_i$ ($i=1,2,3$) and $c,d$.  
Notice that the symplectic genus three surface $\Sigma_3=\Sigma_3 \times \{ pt \}$ is disjoint from the neighborhoods of six lagrangian tori listed above. As a consequence of this, all the relations in (\ref{Luttinger relations for Y_1(m)}) continue to hold in $\pi_1(Y_1(1/p,m/q)\setminus\nu\Sigma_3)$ except the relation $[c,d]=1$. The commutator relation $[c,d]$ is no longer trivial, and it represents a meridian of $\Sigma_{3}$ in $\pi_1(Y_1(1/p,m/q)\setminus\nu\Sigma_3)$. 

Let us choose a diffemorphism $\psi$ that maps the generators of $\pi_1$ as follows:

\begin{equation}\label{eq:psi mapsto for X_1(m)}
a_i \mapsto \bar{a}_i^{\scparallel}, \ \ 
b_i \mapsto \bar{b}_i^{\scparallel}, \ \ 
i = 1,2,3.  
\end{equation}

The manifold $X_1(1, p, q, \psi)$ is symplectic (cf.\ \cite{gompf}). 

\begin{lem}\label{lem:1-2}
The set\/ $S_{5,5}=\{X_1(m, p, q, \psi)\mid m, p, q \geq 1, \psi \}$ consists of $4$-manifolds that all are a fake rational homology $5\CP\# 5\CPb$. Moreover, the family $S_{5,5}$ contains an infinite subsets consisting of pairwise non-diffeomorphic minimal symplectic and non-diffemorphic non-symplectic\/ $4$-manifolds.
\end{lem}

\begin{proof}
We have 
\begin{eqnarray*}
e(X_1(m, p, q, \psi)) &=& e(Y_{1}(1/p,m/q)) + e(M\#\CPb) -2 e(\Sigma_{3}) = 0 + 4 + 8 = 12,\\
\sigma(X_1(m, p, q, \psi)) &=& \sigma (Y_{1}(1/p,m/q)) + \sigma(M\#\CPb) = 0 + 0 = 0.
\end{eqnarray*}

It follows from Seifert-Van Kampen theorem that $\pi_1(X_1(m, p, q, \psi))$ is a quotient of the following group:  

\begin{equation}\label{pi_1(X_1(m))}
\frac{\pi_1(Y_{1}(1/p,m/q)\setminus\nu\Sigma_{3})\ast
\pi_1(M\#\CPb\setminus\nu\bar{\Sigma}_3)}{\langle a_1=\alpha_1,\,
b_1=\alpha_2,\, 
a_2=\alpha_3,\, 
b_2=\alpha_4,\, 
a_3=\alpha_5,\, 
b_3=\alpha_6,\,
\mu(\Sigma_{3})=\mu(\bar{\Sigma}_3)^{-1} \rangle},
\end{equation}

Since any meridian of $\bar{\Sigma}_3$ is trivial in $\pi_1(M\#\CPb\setminus\nu\bar{\Sigma}_3)$, the generator $\mu(\Sigma_{3}) = [c,d]$ coming from (\ref{pi_1(X_1(m))}) is trivial as well. Using the above identification of the fundamental group $\pi_1(X_1(m, p, q, \psi))$, or applying Mayer-Vietoris sequence for the triple $(Y_{1}(1/p,m/q)\setminus\nu\Sigma_{3}, M\#\CPb\setminus\nu\bar{\Sigma}_3, {\Sigma}_3 \times S^1)$, it is straightforward to check that $b_{1}(X_1(m, p, q, \psi)) = 0$. Also, by setting $q = 1$ and variying $p$, introduces $p$ torsion into $H_{1}(X_1(m, p, q, \psi); \mathbb{Z})$.  

To prove minimality for symplectic case (i.e., when $m=1$), we first observe that $Y_1(1/p,1/q)$ is minimal and that the only $-1$ sphere in $M\#\CPb$ is the exceptional sphere $E$\/ of the blow-up.  Since $E$\/ intersects $\bar{\Sigma}_3$ once in $M\#\CPb$, there is no $-1$ sphere in $M\#\CPb \setminus \bar{\Sigma}_3$. It follows from Usher's theorem in \cite{usher} that $X_1(1,p,q,\psi)$ is symplectically minimal.  
Moreover, $-1$ sphere in $M\#\CPb$ and square zero torus in $(Y_{1}(1/p,m/q)$ give rise to $-1$ torus in $X_1(m, p, q, \psi)$. Thus, the intersection form of $X_1(m, p, q, \psi)$ is odd. In symplectic case (i.e., when $m = 1$), $X_1(m, p, q, \psi)$ being odd manifold also follows form the canonical class formula of M. Hamilton (see \cite{Ha}, page 4) for a symplectic fiber sum: $K_{X} = K_{X_{1}} + K_{X_{2}} + \Sigma_{g} + \bar{\Sigma}_{g} + 2(2-2g)B_{X} + \sum_{i=1}^d t_{i}R_{i}$. We refer the reader to \cite{Ha} for unexplained notations.     
  
To show an infinitely many among $X_1(m,p,q,\psi)$'s are pairwise non-diffeomorphic and non-symplectic, we view $X_1(m, p, q, \psi)$ as the result of $5$ Luttinger surgeries and a single $m$\/ torus surgery on $X = (\Sigma_3 \times \mathbb{T}^2)\#_{\psi}(M\#\CPb)$. By Usher's theorem in \cite{usher} that $X$ is symplectically minimal. Next, we compute the Seiberg-Witten invariants of $X_1(m, p, q, \psi)$ and check that infinitely many of them are distinct by applying the same argument as in \cite{ABP, FPS}. Furthemore, we observe that the values of the Seiberg-Witten invariants of $X_1(m, p, q, \psi)$ grow arbitrarily large as $m\rightarrow\infty$. Since the value of the Seiberg-Witten invariant on the canonical class of a symplectic $4$-manifold is $\pm 1$ by Taubes's theorem \cite{taubes}), $X_1(m, p, q, \psi)$ will not be symplectic if $m$\/ is large enough. 

\end{proof}

\begin{rmk} Similar examples can be constructed starting from Cartwright-Steger surfaces $M_{n}$. We first consider the symplectic surfaces that represents the class of canonical line bundle $K_{M_{n}}$ or line bundles with even smaller self-intersection as in \ref{sec: model1}. By blowing up, we reduce the self-intersection of such surface to zero. Next, we take the symplectic connected sum of the resulting manifold with the product manifolds $\Sigma_{g} \times \Sigma_{h}$, and apply Luttinger surgeries along the lagrangian tori. The details of these examples will be discussed in a future paper. 

\end{rmk}

\section{Construction of fake rational homology $(2n-1)\CP\# (2n-1)\CPb$ for $n\geq 4$}
\label{sec:faken}

Let $Y_1(1/p,m/q)$ be the manifold from Section~\ref{sec:fake} with a genus $3$ surface $\Sigma_3$ with self-intersection zero sitting inside it. Let $X_n(m, p, q, \psi)$ = $Y_1(1/p, m/q)\#_{\psi}(M\#\CPb)\#_{id} \cdots \#_{id}(M\#\CPb)$, where $\psi:\partial(\nu\Sigma_3)\to \partial(\nu\bar{\Sigma}_3)$ and there are $n$ copies of $M\#\CPb$. We choose $\psi_{\ast}$ which maps the generators of $\pi_1$ as follows:
\begin{equation}\label{eq:phi mapsto}
a_i \mapsto \bar{a}_i^{\scparallel}, \ \ 
b_i \mapsto \bar{b}_i^{\scparallel}, \ \ 
i = 1,2,3.  
\end{equation}

$X_n(m,p,q,\psi)$ is symplectic when $m=1$. 

\begin{lem}\label{lem:(2n-1)-2n}
For each integer\/ $n\geq 4$, the set\/ $S_{2n-1,2n-1}=\{X_n(m, p, q, \psi)\mid m, p, q \geq 1, \psi \}$ consists of $4$-manifolds that all are fake rational homology $(2n-1)\CP\# (2n-1)\CPb$. Moreover, the family $S_{2n-1,2n-1}$ contains an infinite subsets consisting of pairwise non-diffeomorphic minimal symplectic and non-diffemorphic non-symplectic\/ $4$-manifolds.
\end{lem}

\begin{proof}
We can easily compute that
\begin{eqnarray*}
e(X_n(m, p, q, \psi)) &=& e(Y_{1}(1/p,m/q)) + ne(M\#\CPb) -2 e(\Sigma_{3}) = 0 + 4n + 8 = 4n+8,\\
\sigma(X_n(m,p,q,\psi)) &=& \sigma (Y_{1}(1/p,m/q)) + n\sigma(M\#\CPb) = 0 + 0 = 0.
\end{eqnarray*}

Since the exceptional sphere $E$\/ intersects $\bar{\Sigma}_3$ once in $M\#\CPb$, it follows from Usher's theorem in \cite{usher} that both $n(M\#\CPb)$ and $X_1(m)$ are symplectically minimal.  
Furthermore, sewing $-n$ sphere in $n(M\#\CPb)$ and square zero torus in $(Y_{1}(1/p,m/q)$ gives $-n$ torus in $X_n(m, p, q, \psi)$. Thus, the intersection form of  $X_n(m, p, q, \psi)$ is odd when $n$ is odd.
Applying the same arguments as in the proof of Lemma~\ref{lem:1-2}, we prove the rest of the statements. 
\end{proof}

\begin{rmk}
By taking the normal connected sum $Z_n(m)\#_{\psi}(M\#\CPb)$, we can also obtain an infinite family with the same rational homology $(2n-1)\CP\# (2n-1) \CPb$. 
\end{rmk}

\section{New construction of fake rational homology $3(\mathbb{S}^2\times \mathbb{S}^2)$}
\label{sec:fake3}

To construct a fake rational homology $3(\mathbb{S}^2\times \mathbb{S}^2)$, we proceed as follows. First recall from Section~\ref{sec: model2} that there is a connected genus $3$ symplectic surface $\widetilde{\Sigma}_3$ of self-intersection $0$ in $\Sigma_{2} \times \mathbb{S}^{2}$ that represents the homology class $2[\Sigma_{2} \times \{pt\}]$. Let $M\#\CPb$ be the symplectic $4$-manifold from Section~\ref{sec: model2} with a genus $3$ surface $\bar{\Sigma_3}$ with a self-intersection zero. 

We take the normal connected sum 

\begin{equation*}
X = (\Sigma_{2}\times \mathbb{S}^2)\#_{\psi}(M\#\CPb)
\end{equation*}

We choose $\phi_{\ast}$ which maps the generators of $\pi_1$ as follows:

\begin{equation}\label{eq:phi mapsto}
\tilde{a_i} \mapsto \bar{a}_i^{\scparallel}, \ \ 
\tilde{b_i} \mapsto \bar{b}_i^{\scparallel}, \ \ 
i = 1,2,3
\end{equation}

We have 
\begin{eqnarray*}
e(X) &=& e(\Sigma_{2}\times \mathbb{S}^2) + e(M\#\CPb) -2 e(\Sigma_{3}) = -4 + 4 + 8 = 8,\\
\sigma(X) &=& \sigma (\Sigma_{2}\times \mathbb{S}^2) + \sigma(M\#\CPb) = 0 + 0 = 0.
\end{eqnarray*}

It is easy to see that $X$ has the rational homology of $3(\mathbb{S}^2\times \mathbb{S}^2)$, and it follows form the canonical class formula that $X$ is spin. The details can be filled in as in the previous proofs and are left to the curious reader as an exercise.

\section*{Acknowledgments}  The author greatly benefited from illuminating discussions with Professor Sai-Kee Yeung during his visits to Purdue University in 2010, March 2012, and from various email exchanges with him. The author very grateful to him for drawing our attention to his work on fake projective planes \cite{PY, Ye} and the surfaces of Cartwright-Steger \cite{CS} during these visits. The author also thanks B. D. Park for helpful discussions. The author is partially supported by NSF grants FRG-0244663 and DMS-1005741.

\end{document}